\documentclass{amsart}
\usepackage[numbers]{natbib}
\usepackage[usenames,dvipsnames]{xcolor}
\usepackage{hyperref}
\newtheorem{theorem}{\bf Theorem}[section]

\newtheorem{proposition}{\bf Proposition}[section]

\usepackage{soul}

\newcommand{\RR}{\mathbb{R}}
\newcommand{\CC}{\mathbb{C}}
\newcommand{\NN}{\mathbb{N}}

\newcommand{\A}{\mathbb{A}}
\newcommand{\B}{\mathbb{B}}
\newcommand{\D}{\mathbb{D}}
\newcommand{\I}{\mathbb{I}}
\newcommand{\K}{\mathbb{K}}
\newcommand{\Q}{\mathbb{Q}}

\newcommand{\LL}{\mathbb{L}}

\newcommand{\e}{\varepsilon}

 \def\mG{\mathsf{G}} 
 \def\mV{\mathsf{V}}
\def\mE{\mathsf{E}}
 \def\mv{\mathsf{v}}
 \def\me{\mathsf{e}}

 \newcommand{\sG}{\mathcal{G}}
 \newcommand{\sK}{\mathcal{K}} 
 
 \newcommand{\sL}{\mathcal{L}}
 \newcommand{\sP}{\mathcal{P}}

 \newcommand{\rL}{\mathrm{L}}
\newcommand{\rM}{\mathrm{M}}
\newcommand{\rC}{\mathrm{C}}
\newcommand{\rW}{\mathrm{W}}

\newcommand{\Tt}{(T(t))_{t\ge0}}

\newcommand{\Ttd}{(T'(t))_{t\ge0}}

 \newcommand{\LGp}{\rL^p(\mathcal{G})} 
 \newcommand{\p}{{\raisebox{1.3pt}{{$\scriptscriptstyle\bullet$}}}}

\newcommand{\LpneCm}{\mathchoice{\rL^p\bigl([0,1],\CC^m\bigr)}{\rL^p([0,1],\CC^m)}{\rL^p([0,1],\CC^m)}{\rL^p([0,1],\CC^m)}}
\newcommand{\WepneCm}{{\rW^{1,p}([0,1],\CC^m)}}
\newcommand{\WzpneCm}{{\rW^{2,p}([0,1],\CC^m)}}

\DeclareMathOperator{\diag}{diag}
\DeclareMathOperator{\fix}{fix}
\DeclareMathOperator{\rank}{rank}
\DeclareMathOperator{\ran}{ran}

\title{Semigroups for dynamical processes on metric graphs}

\author[M.~Kramar Fijav\v{z}]{Marjeta Kramar Fijav\v{z}}
\address{Marjeta Kramar Fijav\v{z}, University of Ljubljana, Faculty of Civil and Geodetic Engineering, Jamova 2, SI-1000 Ljubljana, Slovenia / Institute of Mathematics, Physics, and Mechanics, Jadranska 19, SI-1000 Ljubljana, Slovenia}
\email{marjeta.kramar@fgg.uni-lj.si}

\author[A.~Puchalska]{Aleksandra Puchalska}
\address{University of Warsaw, Faculty of Mathematics, Informatics and Mechanics, Institute of Applied Mathematics and Mechanics, Banacha 2, 02-097 Warsaw, Poland}
\email{apuchalska@mimuw.edu.pl}



\thanks{The first author was partially supported by the Slovenian Research Agency, Grant No. P1-0222. The second author's research was supported by National Science Centre, Poland, 2017/25/N/ST1/00787. This article is based upon work from COST Action 18232, supported by COST (European Cooperation in Science and Technology). www.cost.eu.
}

\begin{document}
\begin{abstract}
We present the operator semigroups approach to first- and second-order dynamical systems taking place on metric graphs. We briefly survey the existing results and focus on the well-posedness of the problems with standard vertex conditions. Finally, we show two applications to biological models.
\end{abstract}

\maketitle

\section{Introduction}

Graphs or networks of various kinds are in 21st century omnipresent in everyday's life as well as in science. Graph theory, the field of discrete mathematics that deals with the combinatorial and topological  structure of networks, has experienced a boom in 1950s with the emergence of powerful computers. Since then, it has extremely developed and spread into many other fields, such as operational research, complex networks, or computer algorithms. At the same time,  chemists, physicists, biologists and engineers  started to use networks in modelling. 

In order to model certain dynamical processes along the edges of a graph with appropriate boundary or transmission conditions in the vertices some new mathematical tools from analysis were needed. The first results dealing with heat and wave equations on \emph{metric graphs} (also called \emph{networks} or \emph{1-D ramified spaces}) appeared in the mathematical literature around 1980. In particular, we mention the pioneering work by Lumer~\cite{Lum80}, Roth~\cite{Roth:84}, Ali Mehmeti~\cite{Ali:84}, Nicaise~\cite{Nic:85}, and von Below~\cite{Bel:85}.

In the next two decades, many authors used functional analytic methods to treat such problems, 
let us only list some works: \cite{GOO:93,LLS:94, Ali:94,Cat:97,CF:03,ABN:01,DZ:06,WaKa12}.
Simultaneously, another community was mainly interested in spectral problems associated to the second order - especially Schr\"odinger - equations on a network structure (calling it a \emph{quantum graph}), see e.g.~Exner \cite{Ex:90}, Kottos and Smilansky \cite{KS:97}, Kostrykin and Schrader \cite{KS:99}, Carlson \cite{Car00}, Kuchment \cite{Ku:02,Ku:08}, and Berkolaiko and Kuchment \cite{BK13}.

All the mentioned works are set in $\rL^2$-spaces, the considered operators are all self-adjoint, and the applied methods strongly rely on various Hilbert-spaces techniques. In some cases extrapolation is used  to generalise the results to $\rL^p$-spaces. However, another approach is needed to study problems in Banach spaces, for example to model flows in the $\rL^1$-setting which is appropriate for modelling density of particles. Here, operator semigroups techniques for evolution equations have been proven to be useful. The first model of this kind was proposed by Barletti \cite{Bar96}, but then ten years must have passed before the topic was rediscovered and gained considerable popularity.

The semigroup approach to linear transport equation on finite networks was initiated, independently of Barletti's work, in 2005 by the author and Sikolya  \cite{KS05} and further pursued by
Sikolya \cite{Sik05}, M\'atrai at al.~\cite{MS07},
Kunszenti--Kov\'acs \cite{Kun08a}, and Banasiak et al.~\cite{BN:14,Ban:16}. 
Following the same line, Radl \cite{Rad08}  considered the linear Boltzmann equation with scattering,  Engel et al.~\cite{EKNS08,EKKNS10, EKF:17} and Boulite et al.~\cite{BBEAM:13}  vertex control problems, Kl\"oss~\cite{Klo:12} wave equation, 
Bayazit et al.~\cite{BDR:12,BDK:13} delay and non autonomous transport problems, while
Dorn et al.~\cite{Dor08,DKS09}, Kunszenti--Kov\'acs~\cite{Kun08b}, Namayanja  \cite{Nam:18}, and Budde et al.~\cite{BKF:20} studied 
 transport problems in infinite networks. New insights into the relation between network structure and dynamics were given in \cite{BFN:16a, BF:15}. 
  Small parameter problems for diffusion on networks, initiated by Bobrowski \cite{Bob12}, came straight from a biological application and were further developed in \cite{BFN:16,  BKK:17, Greg:19}. The relation between diffusion on  the edges compared with its counterpart dynamics in the vertices became a motivation to transport analogues given by Banasiak et al.~\cite{BFN:16,Bob:16,BF:17,BP18}. 
 Finally, let us mention some survey publications written over the years: \cite{DFKNR:10,BFT:18, BP19}.
 
By combining semigroup and form methods in $\rL^2$-spaces, diffusion problems were considered in \cite{KMS:07,KKVW09,SV:11,ADK:14, SSJW:15},  hyperbolic problems in \cite{KraMugNic20}, and mixed problems in \cite{HM:13}.
For a thorough display of these methods we refer to the monograph by Mugnolo \cite{Mug:14}. 
Another type of semigroups approach is by applying the theory of port-Hamiltonian systems, see e.g.~\cite{JMZ:15,WauWeg19}.

Aiming to the general, Banach space techniques, we shall present here the perturbation methods from \cite{EKF:19,EKF:20} that allow the study of transport and diffusion processes with non-constant coefficients in general $\rL^p$-spaces. These methods are also suitable for non-compact graphs and yield results for various - also non-local - conditions. For the sake of simplicity, we consider here only compact graphs. In \autoref{sec:preliminaries} we introduce the setting and notations and present two simple generation results for 
first- and second-order problems on metric graphs. We apply these results in \autoref{sec:wp} to transport and  diffusion problems on graphs with so-called standard vertex conditions. For the transport case we also discuss some qualitative properties of the solutions. In  \autoref{sec:applications} we demonstrate 
the usage of the developed theory to the selected real life problems: studies of genetic mutations and synaptic transmissions. 
In this way the impact of the structure of the graph to the dynamic of the relevant biological process gets clear.

\section{Preliminaries}\label{sec:preliminaries}
\subsection{Metric graphs}\label{graphs}

Let $\mG=(\mV,\mE)$ be a simple,  undirected,  finite, connected graph with the set of vertices $\mV=\{\mv_1,\dots,\mv_n\}$ and the set of edges $\mE = \{\me_1,\dots,\me_m\}$.
The structure of  $\mG$ is defined by one of the graph matrices:
\begin{itemize} 
\item the $n\times n$ \emph{adjacency matrix} $\A=(a_{ij})$ giving a vertex to vertex relation, i.e., 
$a_{ij} \ne 0 \iff \mv_i$  and $\mv_j$ are connected by an edge,
\item the $m\times m$ \emph{adjacency matrix of the line graph} $\B=(b_{ij})$ giving an edge  to edge relation, i.e., 
$b_{ij} \ne 0 \iff \me_i$  and $\me_j$ share a common vertex, or
\item the $n\times m$ \emph{incidence matrix}  $\Phi=(\phi_{ij})$ giving a vertex to edge relation, i.e., 
$\varphi_{ij} \ne 0 \iff \mv_i$  is an endpoint of $\me_j$.
\end{itemize}
If the nonzero elements of a graph matrix all equal 1, we say that $\mG$ is an \emph{unweighted} graph, otherwise $\mG$ is \emph{weighted}.
For a vertex $\mv\in\mV$, we denote by $\Gamma(\mv)$ the set of all the  edges in $\mG$ incident to $\mv$ and by $d_\mv:= |\Gamma(\mv)|$ the  \emph{degree} of $\mv$.
We call $\D:=\diag(d_{\mv})$ the \emph{degree matrix} and $\LL:=\D-\A$ the \emph{Laplacian matrix}.

By associating to each edge $\me_k$ an interval, normalised as $[0,1]$ for simplicity, we obtain from the discrete object $\mG$  a metric object $\sG$ called a \emph{metric graph}, that is a collection of intervals with endpoints ``glued" to a network structure. By an abuse of notation we shall denote the vertices at the endpoints of the edge $\me$ by  $\me(0)$ and  $\me(1)$, respectively.
Further, when considering a function $f$ on the edge $\me\equiv [0,1]$,  we shall occasionally write $f(\mv):=f(s)$ if $\me(s)=\mv$ for $s=0$ or $s=1$.

We introduce an orientation of the graph $\sG$ contrary to the parametrisation of the edges as intervals. We thus denote by $\Phi^{-}:=(\phi^{-}_{ij})$ and   $\Phi^{+}:=(\phi^{+}_{ij})$
the $n\times m$ \emph{outgoing} and \emph{incoming} incidence matrix, respectively, defined as
\begin{equation}\label{eq:incident}
 \phi^{-}_{ij} := 
 \begin{cases}
  1, & \text{if } \me_j(1) = \mv_i,\\
  0, & \text{otherwise},
 \end{cases}
 \qquad 
\text{and}
\qquad 
 \phi^{+}_{ij} := 
 \begin{cases}
  1, & \text{if } \me_j(0) = \mv_i,\\
  0, & \text{otherwise}.
 \end{cases}
\end{equation}
If the $i$-th row of  $\Phi^-$ (resp.~$\Phi^+$)  is zero, we say that vertex  $\mv_i$  is a \emph{sink} (resp.~a \emph{source}) of $\sG$.

Let $b^w_{jk} \ge 0$, $1\le j,k\le m$, be some nonnegative \emph{weights}. We shall also use the $m\times m$ (transposed) weighted adjacency matrix of the line-graph $\B_w:=(b^w_{jk})$ defined as
\begin{equation}\label{eq:adj-line}
b^w_{jk} \ne 0 \iff \me_k(0) = \mv_i = \me_j(1)
\end{equation}
and the  $m\times m$ weighted outgoing degree matrix $\D_w^-=\diag(d_{k}^{w-})$  of the line-graph given by
\begin{equation}
d_k^{w-}=\sum_{j=1}^m b^w_{jk}.
\end{equation}
The edges $\me_{j_1},\dots,\me_{j_k}$ forming a cycle in $G$ are a \emph{directed cycle} in $\sG$ if 
\[b^w_{j_i j_{i-1}\ne 0} \,\text{ for } \,i=2,\dots,k\quad\text{and}\quad b^w_{j_1 j_{k}}\ne 0.\]
Finally, we introduce an oriented version of the Laplacian matrix, called the \emph{outgoing Kirchhoff matrix} (cf.~\cite[Def.~2.18]{Mug:14}). We will use it for the line graph, hence we define
\begin{equation}\label{eq:K_matrix}
\sK^-:=\D^-_w-\B_w^\top.
\end{equation}

\subsection{Semigroups, generators, and domain perturbations}

It is well-known that for a linear operator $A:D(A)\subset X\to X$ on a Banach space $X$ an abstract Cauchy problem of the form
\begin{equation}\label{eq:acp}
\begin{cases}
\dot x(t)= A x(t),&t\ge0,\\
x(0)=x_0,
\end{cases}
\end{equation}
is well-posed if and only if $A$ generates a strongly continuous semigroup on $X$,  for details see \cite[Sect.~II.6]{EN:00}.
Consider the Banach space of $\rL^p$-functions, $p\ge 1$, defined on the edges of the metric graph $\mathcal{G}$, 
\begin{equation*}
X=\LGp:=\LpneCm. 
\end{equation*}
Let us  further define
\begin{equation*}
\begin{aligned}
\rW^{1,p}(\sG)&:=\WepneCm,\\
\rW^{2,p}(\sG)&:=\WzpneCm, \\
\rC(\sG)&:=\rC([0,1],\CC^m).
\end{aligned}
\end{equation*}
We are going to study first- and second-order differential operators on $\LGp$ of the form
\begin{equation}\label{eq:generators}
A_1:=c(\p)\cdot \frac{d}{ds}\quad\text{and}\quad A_2 := a(\p)\cdot\frac{d^2}{ds^2},
\end{equation}
respectively. For the coefficients in \eqref{eq:generators} we 
 assume that  $c(\p),a(\p)\colon [0,1]\to M_m(\RR)$
 are bounded Lipschitz continuous matrix-valued functions 
such that 
$c(\p):=\diag({c_i(\p)})$ and $a(\p):=\diag({a_i(\p)})$
with strictly positive diagonal entries: 
\begin{equation}\label{eq:coef-2}
c_{i}(s), a_{i}(s)>0\quad\text{for all }s\in[0,1],\ i=1,\ldots,m.
\end{equation}
Note, that even more general non-diagonal coefficients were allowed in  \cite{EKF:19,EKF:20}.

The structure of the graph  $\mathcal{G}$ is encoded in the boundary conditions appearing in the domains,
\begin{equation}\label{eq:dom-general}
\begin{aligned}
D\left(A_1\right)&:=\bigl\{f\in\rW^{1,p}(\sG) \mid \Psi f=0\bigr\}\quad\text{and}\\
D\left(A_2\right)&:=\bigl\{f\in\rW^{2,p}(\sG) \mid \Psi_0 f=0,\;\Psi_1( f'+ B f)=0\bigr\}
\end{aligned}
\end{equation}
for some linear and bounded ``boundary functionals"  $\Psi,\Psi_0,\Psi_1\colon \rC(\sG)\to\CC^m$ and a ``boundary operator'' $B\colon \LGp\to\LGp$.
The generation results for operators $A_1$ and $A_2$  with domains as in  \eqref{eq:dom-general} were obtained in \cite{EKF:19,EKF:20} by applying the Staffans-Weiss-type of boundary perturbation of the domain developed in \cite{ABE:13} and \cite{HMR:15}. 
Boundary functionals $\Psi,\Psi_0,\Psi_1$ and coefficients $c(\p),a(\p)$ are used to define the so-called ``input-output maps'' ${\mathcal R}_{t_0} \in\sL(\rL^p([0,t_0],\CC^m))$, see \cite[Lem.~2.3]{EKF:20} and \cite[Lem.~2.2]{EKF:19}. Then, it is shown in \cite[Theorem 2.4]{EKF:20} and \cite[Theorem 2.3]{EKF:19}  that the invertibility of ${\mathcal R}_{t_0}$ guarantees that  $(A_1, D(A_1))$ and $(A_2, D(A_2))$ generate $C_0$-semigroups on $\LGp$, respectively. 
Here, we state the generation results just for the special case of boundary conditions  given in terms of  matrices.

\begin{proposition}[\cite{EKF:20}, Corollary 2.16]\label{cor-mat-flow}
Let $V_0, V_1\in\rM_{m}(\CC)$. If 
$\det(V_1)\ne 0$
then the operator 
\begin{equation}\label{eq:A1-Kir}
A_1 = c(\p)\cdot\frac{d}{ds},\quad
D\left(A_1\right)=\left\{f\in\rW^{1,p}(\sG) \mid
V_0f(0) = V_1f(1)\right\}
\end{equation}
generates a $C_0$-semigroup on $\LGp$. If moreover also $\det(V_0)\ne 0$ 
then  we obtain a $C_0$-group.
\end{proposition}
 
\begin{proposition}[\cite{EKF:19}, Corollary 2.11]\label{cor-mat-diff}
For $k_0,k_1\in\NN$ satisfying $k_0+k_1=2m$ let
\[V_0, V_1\in\rM_{k_0\times m}(\CC)\quad \text{and}\quad W_0, W_1\in\rM_{k_1\times m}(\CC).\]
Let $B\in\sL(\LGp,\CC^{k_1})$ and assume that $
B \bigl(\rW^{1,p}(\sG)\bigr) \subseteq\rW^{1,p}([0,1],\CC^{k_1}).
$
If the determinant
\begin{equation*}\label{eq:det}
\det
\begin{pmatrix}
V_1&V_0\\
W_1\cdot a(1)^{-1/2}&W_0\cdot a(0)^{-1/2}
\end{pmatrix}
\ne 0,
\end{equation*}
then the operator 
\begin{equation}\label{eq:A2-mat}
A_2= a(\p)\cdot\frac{d^2}{ds^2},\quad
D\left(A_2\right)=\left\{f\in\rW^{2,p}(\sG) \; \bigg|\; 
\begin{aligned}
&V_0f(0) + V_1f(1)= 0\\[-3pt]
&W_0f'(0) - W_1f'(1)  +(Bf)(0)=0
\end{aligned}
\right\}
\end{equation}
generates an analytic semigroup on $\LGp$  of angle $\frac{\pi}{2}$.
\end{proposition}

\section{Well-posedness and some structural properties}\label{sec:wp}

We present here two simple applications of \autoref{cor-mat-flow} and \autoref{cor-mat-diff} that yield the well-posedness of first- and second-order processes, respectively, on metric graphs with standard vertex conditions.

\subsection{Flows  with standard vertex conditions}

We start by considering a transport process along each edge $\me_j$ of the metric graph $\sG$  given by
\begin{equation}\label{eq:TE}
\frac{\partial}{\partial t}\, u_j(t,s) = c_j(s)\cdot \frac{\partial}{\partial s}\, u_j(t,s)
,\quad t >0,\  s\in(0,1),\quad j=1,\dots,m ,
\end{equation}
where $u_j$ represents the density of the transported material and  $c_j$ is the velocity function satisfying \eqref{eq:coef-2}. Since we have assumed that all  $c_j>0$ we consider the transport on $\me_j\equiv [0,1]$ from the vertex $\me_j(1)$ to the vertex $\me_j(0)$.

In the vertices, the material gets redistributed according to certain rules. A standard assumption is that the process complies with the \emph{Kirchhoff's law}, that is in every vertex, at any time,  the total incoming flow equals the total outgoing flow. Since the flow on each edge is always nonnegative, this means that we may without loss of the generality assume that $\sG$ has no sinks or sources (see also \cite[Thm.~2.1]{BN:14}) . 
The Kirchhoff condition can be written in terms of our incidence matrices as 
\begin{equation}\label{eq:kirch}
 \Phi^{-} c(1) u(t,1) = \Phi^{+}  c(0) u(t,0).
\end{equation}
For the well-posedness of the transport problem on $m$ compact intervals $m$ boundary conditions are needed. Kirchhoff's law \eqref{eq:kirch} gives us $n$ conditions, each row corresponding to the condition in one vertex. The graph with $m=n-1$ edges is a tree and the graph with $m=n$ is unicyclic. Since we assume no sources or sinks, the only graph where Kirchhoff's laws give sufficiently many boundary conditions is a cycle. In all the other cases $m>n$ and we need more conditions. 

A natural further assumption is to prescribe how the material gets redistributed in the vertices. 
Let $w_{ij}$ represent the proportion of the material  that is distributed from vertex $\mv_i$ into edge $\me_j$. We assume that 
\begin{equation}\label{eq:w}
0\le w_{ij}\le 1,\quad w_{ij} \ne 0 \iff \phi_{ij}^- \ne 0,\quad\text{and}\quad \sum_{j=1}^{m} w_{ij} = 1,
\end{equation}
for all $i=1,\dots,n$ and $j = 1,\dots,m$.
For every edge $\me_j$ such that $\mv_i = \me_{j}(1)$ we thus take
\begin{equation}\label{eq:dis}
    c_j(1) u_{j}(t,1)= w_{ij} \left[ \Phi^{+}  c(0) u(t,0)\right]_i.
\end{equation}
This yields the $m$ boundary conditions we need. Note, that \eqref{eq:w} guarantees the conservation of mass  in every vertex and that conditions \eqref{eq:dis} \& \eqref{eq:w} together imply Kirchhoff's law \eqref{eq:kirch}. 

\begin{proposition}\label{prop:flow-wp}
Let $\sG$ be a finite connected metric graph given by the incidence matrices \eqref{eq:incident} with no sinks nor sources. 
Then the  system
\begin{equation}\label{eq:F}
\left\{\begin{array}{rcll}
\frac{\partial}{\partial t}\, u_j(t,s) &=& c_j(s)\cdot \frac{\partial}{\partial s}\, u_j(t,s)
,&t >0,\  s\in(0,1),\\
\phi_{ij}^{-}c_j(1) u_{j}(t,1)&=&w_{ij}\sum_{k=1}^{m}\phi_{ik}^{+}c_k(0) u_{k}(t,0),& t >0,
\\
u_j(0,s)&=& f_j(s), &s\in\left(0,1\right),
\end{array}
\right.
 \end{equation}
where $j=1,\dots,m$, $i=1,\dots,n$, 
is well-posed on  $\LGp$. Its solution is given as 
\[u(t,x)= T(t)f (x)\] where $\Tt$ is a  $C_0$-semigroup on $\LGp$.
\end{proposition}

\begin{proof}
Since there are no sinks, the boundary conditions in  \eqref{eq:F}  are equivalent to 
\begin{equation*}
u(t,1) =  \B_c u(t,0) \quad\text{where}\quad  \B_c:= c(1)^{-1} \B_w c(0)
\end{equation*}
and $\B_w$ is the adjacency matrix defined in \eqref{eq:adj-line} where we take $b^w_{jk} = w_{ij}$ and $\mv_i$ is the common vertex of the edges $\me_k$ and $\me_j$, see \cite[Prop.~18.2]{BKFR:17}. Now, letting
\begin{equation}\label{eq:flow-op}
A_1 := c(\p)\cdot\frac{d}{ds},\quad
D\left(A_1\right) :=\left\{f\in\rW^{1,p}(\sG)  \; \big|\; 
f(1)  =  \B_c f(0) \right\},
\end{equation}
 \autoref{cor-mat-flow} yields that the problem  \eqref{eq:F}  is well posed.
\end{proof}

Let us add some comments to the obtained result. 
Since $\B_w$ can be expressed via incidence matrices (see \cite[(18.3)]{BKFR:17}), it follows from our assumptions that $\rank\B_w =n$. Hence, the matrix $\B_c$ in \eqref{eq:flow-op}  is invertible if and only if $\sG$ is a directed cycle and, by \autoref{cor-mat-flow}, this is the only case when the solution semigroup $\Tt$  is actually a group.

Further, let us remark that we have actually proven a much more general result. The proof of \autoref{prop:flow-wp} gives us the generation property for the 
operator given in \eqref{eq:flow-op} for \emph{any matrix} $\B_c$, not necessarily related to the graph itself! One can thus reverse the question and ask, when is the problem with a general matrix \emph{graph realizable}, that is, when is given matrix $\B_c$ an  adjacency matrix of the line graph of $\sG$. This question was studied in \cite{BF:15}.
\medskip

Under the assumption on the weights  \eqref{eq:w}, the matrix  $\B_w$ is column stochastic. This turns out to be important when studying further qualitative properties of the solutions. 
 Many properties of the solution semigroup are given by the structure of the graph. For example, the semigroup $\Tt$ is irreducible if and only if the oriented metric graph $\sG$ is strongly connected (cf.~ \cite[Prop.~18.16]{BKFR:17} and  \cite[Lem.~4.5]{MS07}). 
To formulate another result of this kind we need some more notations. For every $ j=1, \dots, m$ we define 
\begin{equation}\label{eq::def-varphi}
\varphi_j(s):=\int_0^s\frac{dr}{c_j(r)}
\quad \text{ for }s\in[0,1].
\end{equation}
The following  condition plays a crucial role in the long-term behaviour of the solutions.
\begin{equation}\label{ldq}
\begin{aligned}
\text{There exists } &0<d\in\RR \text{ such that }d\cdot (\varphi_{j_1}(1) + \cdots + \varphi_{j_k}(1))\in\NN \\ 
&\text{for all directed cycles }e_{j_1},\dots,e_{j_k} \text{ in }\sG.
\end{aligned}
\end{equation}
We call a subgraph $\sG_r$ of $\sG$ a \emph{terminal strong component} if it is strongly connected and there are no outgoing edges of $\sG_r$, see  \cite[page 17]{BG:09}.

\begin{theorem} \label{thm:flow-asy} 
Let $\sG$ be a connected graph  with terminal strong components $\sG_{1},\dots,\sG_{{\ell}} $  and $\Tt$ a semigroup associated with the transport problem \eqref{eq:acp} -- \eqref{eq:flow-op}. Then  the space $\LGp$ and the semigroup $\Tt$ can be decomposed as
$$\LGp=X_{n} \oplus X_{s} \oplus X_{r_1} \oplus \cdots \oplus X_{r_{\ell}}\quad \text{and}\quad T(\p)=T_{n}(\p) \oplus T_{s}(\p) \oplus T_{r_1}(\p)\oplus \cdots \oplus T_{r_{\ell}}(\p)$$
such that  all the subspaces in the decomposition are $T(t)$-invariant  and the following holds.\vspace{-.3cm}
\begin{enumerate}
\item  $T_n(\p)$ is nilpotent on $X_{n}$.
\item  $T_s(\p)$ is strongly stable on $X_{s}$.
\item If for some $1\le i\le\ell$ the graph $\sG_{i}$  satisfies Condition \eqref{ldq}  then $T_{r_i}(\p)$ is a periodic irreducible group on $X_{r_i}$ with period 
\[ \tau_i = \frac{1}{d}\gcd\left\{ d\cdot (\varphi_{j_1}(1) + \cdots + \varphi_{j_k}(1)) \mid e_{j_1},\dots,e_{j_k}  \text{ is a directed cycle in }\sG_i\right\}.\]
 \item If for some $1\le i\le\ell$ graph $\sG_{i}$ does not satisfy Condition \eqref{ldq}   then $T_{r_i}(\p)$ converges strongly
towards a projection onto the one-dimensional subspace $X_{r_i}$.
\end{enumerate}
\end{theorem}
\begin{proof}
For simplicity, first assume that all coefficients $c_j(\p)\equiv c_j$ are constant. The decomposition of the space $\LGp$ is obtained as the spectral decomposition for the semigroup generator which corresponds to the decomposition of the adjacency matrix of the (line) graph according to the graph structure, as explained in the proofs of \cite[Thm.~4.10]{KS05} and  \cite[Thm.~5.1 \& Thm.~5.2]{BN:14}. The behaviour of the semigroups $T_{r_i}(\p)$ corresponding to the terminal strong components  $\sG_i$ is further described in \cite[Thm.~18.19]{BKFR:17}. Finally, considerations for arbitrary coefficients  can be found in \cite[Thm.~4.14 \& Thm.~4.22]{MS07}. 
\end{proof}

\subsection{Diffusion with standard vertex conditions}

Let us now consider the diffusion process along the edges of the metric graph $\sG$ given by
\begin{equation}\label{eq:net-dif}
\frac{\partial}{\partial t}\, u_j(t,s) = a_j(s)\cdot \frac{\partial^2}{\partial s^2}\, u_j(t,s),\quad t >0,\  s\in(0,1),\quad j=1,\dots,m,
\end{equation}
for some variable diffusion coefficients $a_j$ satisfying \eqref{eq:coef-2}. Having  the heat equation in mind, $u_j$ represents the temperature distribution along the edge $\me_j$ and it is reasonable to assume that $u$ is a continuous function on the graph, that is
\begin{equation}\label{eq:cont}
u_j(t,\mv) = u_k(t,\mv)\quad\text{whenever} \quad \me_j,\me_k\in\Gamma(\mv), \text{ for all }\mv\in\mV.
\end{equation}
This continuity condition can be expressed with matrices in the following way. 
For each vertex $\mv$ with degree $d_{\mv} >1$ define the $(d_{\mv} -1)\times d_{\mv}$ matrix
\begin{equation}\label{eq:Iv}
I_{\mv}:=\begin{pmatrix}
1 & -1 &&\\
&\ddots&\ddots&\\
&&1&-1
\end{pmatrix}.
\end{equation}
If the set of edges incident to $\mv$ equals $\Gamma(\mv)=\{\me_{j_1}, \dots, \me_{j_{d_{\mv}}}\}$ and $f(\mv):=(f_{j_1}(\mv), \dots,f_{j_{d_{\mv}}}(\mv))$  is the vector of the values of a function $f\in\rC(\sG)$ at the corresponding endpoints, then the equation
\begin{equation}\label{eq:cont-v}
I_{\mv}  f(\mv) = 0
\end{equation} 
yields the continuity of $f$ in $\mv$. Assuming continuity in all the vertices thus yields together $\sum_{\mv\in\mV} (d_{\mv} -1) = 2m -n$ boundary conditions. In order to obtain a well-posed diffusion problem on $m$ compact intervals (i.e., edges of the graph) we additionally need  $n$ boundary conditions. 

The next standard assumption is to impose in every vertex the Kirchhoff conditions for the heat fluxes. 
Again, this conditions can be written in terms of  incidence matrices as 
\begin{equation}\label{eq:eq:kirch-D}
\Phi^{-}a(1) \frac{\partial}{\partial s}\, u(t,1) =  \Phi^{+}a(0) \frac{\partial}{\partial s}\, u(t,0) 
\end{equation}
yielding the missing $n$ boundary conditions. 

\begin{proposition}\label{prop:diff-wp}
Let $\sG$ be a finite connected metric graph characterized by the incidence matrices \eqref{eq:incident}. 
Then the  system
\begin{equation}\label{eq:D}
\left\{\begin{array}{rcll}
\frac{\partial}{\partial t}\, u_j(t,s) &=& a_j(s)\cdot \frac{\partial^2}{\partial s^2}\, u_j(t,s),& t >0,\  s\in(0,1),\\
u_j(t,\mv_i) &=&u_k(t,\mv_i),&t > 0,\ \me_j,\me_k\in\Gamma(\mv_i),\\
 \sum_{k=1}^{m}\phi_{ik}^{-}a_k(1) \frac{\partial}{\partial s}\, u_{k}(t,1)&=& \sum_{k=1}^{m}\phi_{ik}^{+}a_k(0)  \frac{\partial}{\partial s}\, u_{k}(t,0),& t >0,\\
u_j(0,s)&=& f_j(s), & s\in\left(0,1\right),
\end{array}
\right.
 \end{equation}
 where $j=1,\dots,m$, $i=1,\dots,n$, 
is well-posed on  $\LGp$. 
\end{proposition}

\begin{proof} We will apply \autoref{cor-mat-diff}. To this end we need to write the boundary conditions in terms of some suitable boundary matrices. Let $k_0:= 2m-n$ and $k_1:=n$. We define the matrices $V_0,V_1\in M_{k_0\times m} (\CC)$ as a composition of $n$ blocks. The $i$-th block is $(d_{\mv_i}-1)\times m$ matrix associated with vertex $\mv_i$ and obtained from matrix $I_{\mv_i}$, defined in \eqref{eq:Iv}, as follows.
Each row of $I_{\mv_i}$ has exactly two nonzero entries, 1 and $-1$,  corresponding to a pair of edges $\me_j,\me_k \in \Gamma(\mv_i)$, respectively.  Now rearrange these two entries in the corresponding row of the matrices $V_0,V_1$ taking into consideration the parametrization of edges $\me_j,\me_k$. Namely, in the case $\mv_i = \me_j(0)$ (resp.~$\mv_i=\me_j(1)$) put 1 in the $j$-th column of $V_0$ (resp.~$V_1$) while in the case $\mv_i = \me_k(0)$ (resp.~$\mv_i=\me_k(1)$) put $-1$ in the $k$-th column of $V_0$ (resp.~$V_1$). Now observe that by \eqref{eq:cont-v},
\[V_0 u(t,0) + V_1 u(t,1)=0\]
which yields the continuity condition \eqref{eq:cont} in all the vertices of the graph.

Taking $W_0:= \Phi^{+}a(0)$ and $W_1:=\Phi^{-}a(1) $ the Kirchhoff conditions  \eqref{eq:eq:kirch-D} are expressed by
\[W_0 \frac{\partial}{\partial s}\, u(t,0)  - W_1 \frac{\partial}{\partial s}\, u(t,1)  = 0.\]
Thus, we can rewrite the problem \eqref{eq:D} as an abstract Cauchy problem for the operator $(A,D(A))$ defined in \eqref{eq:A2-mat} and the well-posedness is obtained once we verify that
\begin{equation*}
\det M:=\det
\begin{pmatrix}
V_1&V_0\\
\Phi^{-} a(1)^{1/2}& \Phi^{+} a(0)^{1/2}
\end{pmatrix}
\ne 0.
\end{equation*}
Observe that each column of $M$ corresponds  to exactly one endpoint of an edge. Moreover,  by permuting rows and columns of $M$ we can obtain the block diagonal matrix consisting of $n$ blocks of size $d_{\mv_i}\times d_{\mv_i}$  where each block corresponds to a `'vertex cluster'' - that is one vertex and appropriate endpoints of  its incident edges. We denote by  $M_{\mv}$ the block corresponding to vertex $\mv$. If $d_{\mv} =1$,   $M_{\mv}=\sqrt{a_{j}(\mv)}$ for $\me_j\in\Gamma(\mv)$, otherwise it consists of the matrix $I_{\mv}$ which we complement with the part of appropriately permuted  row of the matrix $\begin{pmatrix}\Phi^{-} a(1)^{1/2} & \Phi^{+} a(0)^{1/2}\end{pmatrix}$ corresponding to the relevant endpoints of the edges $\Gamma(\mv)=\{\me_{j_1}, \dots, \me_{j_{d_{\mv}}}\}$. This way we obtain 
\[M_{\mv} =  \begin{pmatrix}
1 & -1 &&\\
&\ddots&\ddots&\\
&&1&-1\\
\sqrt{a_{j_1}(\mv)}&\dots&\dots &\sqrt{a_{j_{d_\mv}}(\mv)}
\end{pmatrix} \text{ with }\det M_{\mv}  = \sqrt{a_{j_1}(\mv)}+\cdots + \sqrt{a_{j_{d_\mv}}(\mv)}\ne 0. \]
\end{proof}

Since the determinant condition in \autoref{cor-mat-diff} does not depend on the operator $B$, in the same way as above we obtain  the well-posedness of the diffusion problem with the so-called $\delta$-type conditions, see \cite[Sec.~3.3]{EKF:19}.

\section{Graph structure impact on dynamics}\label{sec:applications}

In this section we present two biological models chosen in the way that the first one fits to the network transport theory with standard vertex conditions whereas the second one is modelled with diffusion on the graph with generalised boundary conditions. Semigroup considerations allow us to characterise the dynamical properties of systems including also the relation between asymptotic behaviour and  the graph structure. 

\subsection{A genetic mutation model}\label{Kirchhoff}
Following \cite{BFN:16}, consider a population of cells divided into $m$ compartments according to  their genetic code. We describe the evolution of this population by taking two features into consideration: the normalised age $x\in[0,1]$ of the cell and the specified genetic characteristics $j\in\{1,\dots,m\}$. By $u_j(x,t)$ we denote the density of cells of type $j$ at age $x$ at time $t$. Assume additionally, that this characteristics can be different for a daughter and its mother-cell. Standard cell differentiation in mitosis is described by the matrix $\K=(k_{ij})_{i,j=1}^m$ and rare errors, causing a mutation of the genotype, are denoted by $\Q=(q_{ij})_{i,j=1}^m$. By $k_{ij},q_{ij}\geq 0$ we understand the fraction of mother cells with genetic feature $j$, having daughter cells of type $i$. We describe the general pattern of the proliferation of the genetic characteristic using the model \eqref{eq:acp} -- \eqref{eq:flow-op} in $\rL^1(\mathcal{G},\RR^m)$ with operator $\B_w=\K+\Q$ and $c\equiv 1$. In the whole \autoref{sec:applications} we assume  $X=\rL^1(\mathcal{G}) = \rL^1(\mathcal{G},\RR^m)$ is a real Banach space.

Note that $d_{ij},k_{ij}$ are -- as fractions of cell mass -- nonnegative. By \autoref{prop:flow-wp}, the problem is well-posed and, by \cite[Thm.~3.1]{BFN:16a}, also biologically meaningful since it attains a positive solution for positive initial data. Conservation of mass during reproduction indicates that \eqref{eq:w} holds and therefore $1\in \sigma(\B_w)$.

We assume that any type $i$ of genetic code can be attained which shall entail a connectedness, but not strong connectedness, of the graph $\sG$. Since condition \eqref{ldq} is satisfied for $c\equiv 1$, \autoref{thm:flow-asy} shows that the edges of the graph $\sG$ can be divided into two disjoint groups: the terminal strong components $\sG_t=\bigcup_{i=1}^{\ell}\sG_i$ and the acyclic part $\sG_a=\sG \setminus \sG_t$. The part $\sG_a$ consists of the edges on which the flow vanishes after some time and therefore is strictly related with the number of sources in $\sG$ and with the multiplicity of $0$ in $\sigma(\B_w)$. This part of the network can be interpreted as mutations that occurred in the past but due to the evolution process are not  observed nowadays. The subgraphs $\sG_t$ are related with the eigenvalue $1\in\sigma(\B_w)$ and its multiplicity indicates the number of strongly connected subgraphs in the limit, see \cite[page~17]{BG:09}. Note that if the matrix $\B_w$ is imprimitive then the limit behaviour of the system is periodic with period
$$\tau=\text{lcm}\left\{\tau_i\mid i=1,\dots,\ell\right\}$$
which means that we should observe time fluctuations in the number of cells having specified genotype. For a primitive matrix $\B_w$, the number of cells should stabilise at a certain level even though all the terminal strong components of $\sG$ consist of cycles. For more details of this considerations we refer to the explicit formulae of projection onto the eigenspace of $\B_w$ associated with eigenvalue $1$ computed in \cite[Thm.~3.1]{Ban:16}.

Finally, it is worth mentioning that the long term dynamic acts on the space of notably smaller dimension than $m$, namely, on the eigenspace associated with the Perron eigenvector of $\B_w$. It does not mean however that there are smaller number of mutations involved. 
\medskip

The system \eqref{eq:acp} -- \eqref{eq:flow-op} is considerably rich of information. As a model with both age- and gene- structure it consists of two time scales. Age characterises a cell lifetime which is significantly shorter than the time in which we can observe evolutionary gene mutations. In order to reduce the complexity of the system one can neglect the age-structure but then it is necessary to reflect on how the mutations observed in micro-scale influence the macro-description. For $\e>0$ consider a family of Cauchy problems \eqref{eq:acp} -- \eqref{eq:A1e-Kir}, with 
\begin{equation}\label{eq:A1e-Kir}
A^{\e}:= \frac{1}{\e}\frac{d}{ds},\quad
D(A^{\e}):=\left\{f\in\rW^{1,1}(\sG) \mid \left(\K + \e \Q\right) f(0) = f(1)\right\}.
\end{equation}
This evolution process describes the fast ageing with little number of mutations during mitosis compared to the total number of offsprings. 

Define now two mappings $\Pi_1, \mathcal{P}:\rL^1(\mathcal{G}) \rightarrow \rL^1(\mathcal{G})$,  
\begin{equation}\label{eq:proj}
\Pi_1 u:= \left(e_l\cdot u\right)e_r\quad\text{and}\quad \sP u:= \int_0^1u(x)dx,\quad u\in \rL^1(\mathcal{G}),
\end{equation}
where $e_l$ and $e_r$ are the left and right  eigenvector of $\K$ associated with $\lambda=1$ and normalised so that $e_l \cdot e_r=1$. Let  $\I$ denotes the $m\times m$ identity matrix. Note that $\Pi_1|_{\RR^m}$ is the spectral projection onto the eigenspace $\ker(\I-\K)$ along $\ran(\I-\K)$ while $\sP$ is a projection onto the finite dimensional subspace $\RR^m\subset \rL^1(\mathcal{G})$. Here and in the following, $\RR^m$ is considered either as linear space $(\RR^m,\left\|\cdot\right\|_{\ell_1})$ or as a linear subspace of $(\rL^1(\mathcal{G}),\left\|\cdot\right\|_{\rL^1(\mathcal{G})})$,  that is the subspace of the edge-wise constant functions on the graph, which does not cause an ambiguity.

In the following results we shall use the facts that $\K$ is contractive and $\lambda=1$ is its semisimple eigenvalue, see \cite[Eq.~(17),~Rem.~1]{BP18}. For the considered biological model they are naturally satisfied.

\begin{theorem}\label{flow-Kir-conv}\cite[Cor.~1\&3, Thm.~4.1]{BP18}
For any $\e>0$, let $u_\e(t) = T_\e(t)x_0$ for $x_0 \in \rL^1(\mathcal{G})$ be a solution of \eqref{eq:acp} -- \eqref{eq:A1e-Kir}. If $u(t)=T(t)u(0)$ is a matrix semigroup solution in $\RR^m$ of the problem
\begin{equation}\label{eq:flow-Kir-lim}
\begin{cases}
\dot{u}(t)=\Pi_1 \Q \Pi_1 u(t), &t > 0,\\
u(0)=\Pi_1 \mathcal{P} x_0,&
\end{cases}
\end{equation}
then the following results hold.
\begin{enumerate}
	\item For any $x_0\in \Pi_1 \RR^m$,
		\begin{equation}\label{eq:flow-Kir-conv1}
		\lim\limits_{\e \to 0^+}\|u_\e(t) - u(t)\|_{\rL^1(\mathcal{G})} =0\quad \text{almost uniformly on }[0,\infty).
		\end{equation}
	\item If, additionally,  $\K$ is primitive then, for any $x_0\in \RR^m$, the convergence in \eqref{eq:flow-Kir-conv1} is almost uniform on $(0,\infty)$.	
	\item For any $x_0\in \rL^1(\mathcal{G})$,
		\begin{equation}\label{eq:flow-Kir-conv2}
		\lim\limits_{\e \to 0^+}\|\Pi_1 \sP u_\e(t) - u(t)\|_{\ell_1} =0\quad \text{almost uniformly on }[0,\infty).
		\end{equation}
\end{enumerate}
\end{theorem}

The three types of convergence in \autoref{flow-Kir-conv} show the relation between the micro-model and its aggregated counterpart defined in \eqref{eq:flow-Kir-lim}. The lack of convergence for any $x_0\in \rL^1(\mathcal{G})$, see the counterexample in \cite[Sec.~3]{BF:17}, indicates that the micro-description is richer in information than the macro-model which agrees with intuition. Simultaneously in both approaches, the macro-parameters of the system, namely a total masses at the moment $t\geq 0$, are comparable according to \eqref{eq:flow-Kir-conv2}. 

Using the interpretation of  the projection $\Pi_1$ and condition \eqref{eq:flow-Kir-conv1}, we conclude that the solution to \eqref{eq:flow-Kir-lim} does not approximate the mass at each edge, but rather the total mass concentrated on the terminal strong components $\sG_t$ of the graph $\sG$. If there are, say, $\ell$ such strong components, then the limiting system of ordinary differential equations consists of $\ell$ differential equations describing the evolution of the material trapped in each terminal component. This  goes in line with the long time behaviour of the system given in \autoref{thm:flow-asy}. In other words, the aggregation method presented in \autoref{flow-Kir-conv} yields a macro-model approximating the long term dynamics of the given micro-model. 

Note, finally, that the gene evolution in the aggregated model \eqref{eq:flow-Kir-lim} is embedded twofold: by the Perron eigenvector $e_r$, see \eqref{eq:proj}, giving the long term profile of the flow and by the matrix of mutations $\Q$ which influences the time evolution of the total mass. For details see \cite[Exam.~6]{BP18}.

\subsection{A synaptic transmission model}\label{Robin}

Using the mathematical approach from \cite{Bob12,BFN:16}, we now describe the process of nervous system response to stimulus by modelling a transmission of an information among neurons through a chemical substance called neurotransmitter. The neurotransmitters are stored in the synaptic vesicles situated in axon terminal which, for the need of this model, we subdivide into certain number of compartments called \emph{pools}.
In this approach we assume that the storage in the vesicles and its release to another pool is described by a diffusion in the cytoplasm and its transfer through a semi-permeable membrane. For the sake of simplicity, the spatial distribution of each synaptic pool is represented by an interval $[0,1]$. Hence, the function $u_i(x,t)$ describes the concentration of vesicles in $i$-th pool in position $x\in[0,1]$ at time $t\geq 0$. We follow the concept of Aristizabal and Glavinovi\v c who initiated this considerations in \cite{AG:04}. The dynamics of the densities $u_i$ was modelled in the tree pool case - with large, small, and immediately available pools - similarly to voltages across the capacitors in an electric circuit. This allowed obtaining the rates of transfer between adjacent pools. 

Let us consider the connections between $m$ synaptic pools using a simple, strongly connected and oriented metric graph $\sG$. Let $l_i, l_{ij}$ (resp. $r_i, r_{ij}$) be the rates at which the substance leaves $\me_i$ by vertex $\me_i(1)$ (resp. by $\me_i(0)$) or enters $\me_j$ from $\me_i(1)$ (resp. from $\me_i(0)$). Clearly, all the rates among adjacent edges are positive. The weighted outgoing adjacency matrix $\B_w^-=(b_{ij}^-)$ and the outgoing degree matrix $\D_w^-=(d_{ij}^-)$ of the line graph are given by
\begin{equation}\label{eq:A&D}
b_{ij}^-=l_{ij}+r_{ij},\qquad 
d_{ij}^-=\begin{cases}
    r_i+l_i,&\text{for }i=j,\\ 
   0,   & \text{otherwise}.
\end{cases}
\end{equation}
By Fick's law we obtain vertex conditions of the form
\begin{equation}\label{eq:D_bound1}
\begin{pmatrix} {f'(0)}\\{f'(1)} \end{pmatrix} = \mathbb{K}\begin{pmatrix} {f(0)}\\{f(1)} \end{pmatrix}\quad\text{with}\quad
\K=\begin{pmatrix} \K^{00}&\K^{01}\\ \K^{10}&\K^{11} \end{pmatrix}
\end{equation}
where the matrices 
$\K^{pq}=\left(k_{ij}^{pq}\right) \in M_m(\RR)$, $p,q=0,1$, are defined by
\begin{equation}\label{eq:K_def2}
k^{0q}_{ij}:=\begin{cases}
    -r_{i}&\text{if } i=j,\, q=0,\\
    r_{ij}&\text{if } \me_i(0)=\me_j(q), \, q=0,1, \\
  0   & \text{otherwise},
\end{cases}
\quad
k^{1q}_{ij}:=\begin{cases}
l_{i}&\text{if } i= j,\, q=1,\\
-l_{ij}&\text{if } \me_i(1)=\me_j(q), \, q=0,1,\\
0&\text{otherwise.}\
\end{cases}
\end{equation}
For the details of this construction we refer to \cite[Exam.~3.1]{BFN:16}, with the restriction that in this paper a reverse parametrisation of the interval is considered. 

We can thus rewrite  the model in terms of the Cauchy problem \eqref{eq:acp} -- \eqref{eq:D_operator} with 
\begin{equation}\label{eq:D_operator}
A:= \frac{d^2}{ds^2},\quad
D\left(A\right)=\left\{f\in\rW^{2,p}(\sG) \mid f \text{ satisfies } \eqref{eq:D_bound1}\right\}.
\end{equation}
The existence and uniqueness of the solution of this problem follows directly from \autoref{cor-mat-diff} by choosing $a(\p)\equiv \mathbf{1}\in\RR^m$, $k_0=0$, $k_1=2m$, 
\begin{eqnarray}\label{eq:D_bound2}
W_0:= \begin{pmatrix} -Id\\ 0\end{pmatrix}, \quad W_1:= \begin{pmatrix} 0\\ Id \end{pmatrix},\quad\text{and}\quad B:=\mathbb{K} \begin{pmatrix} Id\\ \psi\end{pmatrix},
\end{eqnarray}
where $\psi f (s) := f(1-s).$
Other practical properties such as positivity or conservation of mass in the process are presented below.

\begin{proposition}\label{prop:prop_e^D}
Let $\Tt$ be the solution semigroup in $\rL^1(\mathcal{G})$ of the problem \eqref{eq:acp} -- \eqref{eq:D_operator}.
Then the following results hold.
\begin{enumerate}
\abovedisplayshortskip=-\baselineskip 
	\item $\Tt$ is a positive semigroup.
	\item $\Tt$ is a Markov semigroup if and only if for any $i=1,\dots,m$
		\begin{equation}\label{eq:dif-contr}
		\sum_{k=1}^m l_{ik}=l_i\quad \text{and}\quad \sum_{k=1}^m r_{ik}=r_i.
		\end{equation}
	\item If $\K$ satisfies \eqref{eq:dif-contr} then $\Tt$ is an irreducible semigroup.
		\end{enumerate}
\begin{proof}
Assertion (i) follows from \eqref{eq:K_def2} and \cite[Cor.~2.6]{BFN:16a} while (ii) is stated in \cite[Exam.~3.1, eq. (3.48)]{BFN:16}. It remains to prove (iii). By (i) and  (ii), $\Tt$ is a positive, Markov semigroup. We first show that  $\Tt$ is also mean ergodic, for a definition see \cite[Def.~V.4.3]{EN:00}, which is for bounded $C_0$-semigroups by \cite[Thm.~V.4.5]{EN:00} equivalent to the condition 
\begin{equation}\label{eq:d_ker}
\fix \Tt=\text{ker} A \quad \text{separates}\quad \fix \Ttd= \text{ker}  A^* .
\end{equation} 
The irreducibility of $\Tt$ then follows  analogously as in the proof of \cite[Thm.~5.1]{KMS:07}.

Note that by \cite[Exer.~II.4.30(4)]{EN:00},  $A$ is resolvent compact so it has only a point spectrum.  Now, $\ker A$ consists of functions $f(x)=C_1x+C_2$, $C_1,C_2\in \RR^m$, satisfying the boundary condition \eqref{eq:D_bound1} which implies
\begin{equation}\label{eq:spA}
\mathbb{M}\begin{pmatrix}C_1\\ C_2\end{pmatrix}:=\begin{pmatrix}\I-\K^{01}&-\K^{00}-\K^{01}\\
\K^{11}-\I&\K^{10}+\K^{11}\end{pmatrix}\begin{pmatrix}C_1\\ C_2\end{pmatrix}=0.
\end{equation}
By \eqref{eq:dif-contr}, $(C_1,C_2)^\top=(\mathbf{0},\mathbf{1})$, where $\mathbf{0}=(0,\dots,0)^\top$ and $\mathbf{1}=(1,\dots,1)^\top$, fullfils \eqref{eq:spA}, therefore $\rank \mathbb{M}\leq 2m-1$. To show that in fact equality holds note that
$$\rank \mathbb{M}\geq\text{rank}\begin{pmatrix}-\K^{00}-\K^{01}\\
\K^{10}+\K^{11}\end{pmatrix}=2m-1.$$
Indeed, define
\begin{equation}\label{eq:K_def}
\sK^-=\K^{10}+\K^{11}-\left(\K^{00}+\K^{01}\right),
\end{equation} 
which by \eqref{eq:A&D} and \eqref{eq:K_matrix}, is an outgoing Kirchhoff matrix of the line graph of $\sG$. By \cite[Lem.~2.13]{Mug:14} the  algebraic multiplicity of $0$ in $\sigma(\sK^-)$ coincides with the number of connected components of $\sG$, so by strong connectedness of $\sG$, $\ker \sK^-=\text{lin}\left\{\mathbf{1}\right\}\cong\ker A$.

We now compute the dual operator to $(A,D(A))$ in $\rL^{\infty}(\mathcal{G})$ as
\begin{equation}\nonumber
 A^*  = \frac{d^2}{dx^2},\quad
D\left( A^* \right)=\left\{g\in \left(\rW^{2,1}(\sG)\right)^* \;\bigg|\; \begin{pmatrix} {g'(0)}\\{g'(1)} \end{pmatrix} = \mathbb{K}^* \begin{pmatrix} {g(0)}\\{g(1)}\end{pmatrix}\right\},
\end{equation}
with $\K^*$ defined in \cite[Sec.~3, eq. (3.2)]{BFN:16}. Since by \cite[Prop.~IV.2.18]{EN:00} the spectra of $A$ and $ A^* $ coincide, an analogous reasoning leads to the conclusion that 
$\ker A^*$ is one dimensional as well.
Note,  that for the dual problem instead of the outgoing Kirchhoff matix $\sK^-$ we choose the  incoming one: $\sK^+:=\left(\sK^-\right)^\top$.
It is now easy to see that  \eqref{eq:d_ker} holds.
\end{proof}
\end{proposition}

Using \autoref{thm:flow-asy} we show that also in the network diffusion process a long time behaviour lumps mass in the strong components of a graph. We obtain also a new type of information which relates the rate of the norm convergence with the  network structure. 
\begin{theorem}\label{thm:dif-conv}
Let $u(t)=T(t)x_0$ for $x_0\in \rL^1(\mathcal{G})$ be the semigroup solution of \eqref{eq:acp} -- \eqref{eq:D_operator}.
If the entries of $\K$ satisfy \eqref{eq:dif-contr} 
 then
\begin{equation}\label{eq:dif-conv}
		\lim\limits_{t \rightarrow \infty} \|T(t) - \Pi \| =0\quad \text{for all }t\geq 0,
		\end{equation}
where $\Pi$ is the strictly positive projection onto $\ker A$, the one-dimensional subspace spanned by $\mathbf{1}$.

Further, let $\lambda$ be the largest non-zero eigenvalue of $A$. Then for any $\e>0$ there exists $M>0$ such that 
\begin{equation}
\left\|T(t) - \Pi\right\|\leq Me^{(\e+\lambda)t} \quad \text{for all } t\geq 0.
\end{equation}
\end{theorem}
\begin{proof}
By \autoref{cor-mat-diff} and \autoref{prop:prop_e^D}, $\Tt$ is a positive, irreducible, analytic semigroup of contractions. 
From the proof of \autoref{prop:prop_e^D} it further follows that $A$ is resolvent compact, $s(A)=0\in\sigma(A)$,  and $\ker A$ is one-dimensional, spanned by $\mathbf{1}$. Therefore, $\Tt$  is also eventually norm continuous (cf.~\cite[Ex.~II.4.21]{EN:00}) and  eventually compact (cf.~\cite[Lem.~II.4.28]{EN:00}). The first assertion now follows by  \cite[Cor.~V.3.3]{EN:00} and \cite[Prop.~14.12]{BKFR:17},  while the second is a consequence of \cite[Cor.~V.3.2]{EN:00}. 
\end{proof}

In analogy to the considerations in \autoref{sec:applications}\ref{Kirchhoff}, we identify now two time scales for the described process of information transmission. Diffusion in the synaptic pools occurs on a millisecond time scale. Therefore, to model synaptic depression in longer time interval, such as a second, we can consider fast diffusion with slow rates of change between synaptic pools. For $\e>0$ consider the family of operators 
\begin{equation}\label{eq:A2e_R}
A^{\e} = \frac{1}{\e}\frac{d^2}{dx^2},\quad
D\left(A^{\e}\right)=\left\{f\in\rW^{2,1}(\sG) \;\bigg|\; \begin{pmatrix} {f'(0)}\\{f'(1)} \end{pmatrix} = \e \mathbb{K}\begin{pmatrix} {f(0)}\\{f(1)} \end{pmatrix}\right\}.
\end{equation}

\begin{theorem}
For any $\e>0$, let $u_\e(t) = T_{\e}(t)x_0$, $x_0 \in \rL^1(\mathcal{G})$, be  the semigroup solution to \eqref{eq:acp} -- \eqref{eq:A2e_R}. If $u(t)=T(t)u(0)$ is a matrix semigroup solution in $\RR^m$ of the problem
\begin{equation}\label{eq:dif-lim}
\begin{cases}
\dot{u}(t)=\sK^- u(t), &t > 0,\\
u(0)=\mathcal{P} x_0,&
\end{cases}
\end{equation} 
for $\sK^-$ and $\sP$ defined in \eqref{eq:K_def} and \eqref{eq:proj}, respectively, then for any $x_0\in \RR^m$
		\begin{equation}\label{eq:dif-conv1}
		\lim\limits_{\e \to 0^+}\left\|u_\e(t) - u(t)\right\|_{\rL^1(\mathcal{G})} =0\quad \text{almost uniformly on }[0,\infty).
		\end{equation}
Additionally, for any $x_0\in \rL^1(\mathcal{G})$, the convergence in \eqref{eq:dif-conv1} holds almost uniformly on $(0,\infty)$.
\end{theorem}
\begin{proof}
\cite[Thm.~3.2]{BFN:16} states that 
\begin{equation}
\lim\limits_{\e \to 0^+}\left\|u_\e(t) - u(t)-w\left(\frac{t}{\e}\right)\right\|_{\rL^1(\mathcal{G})} =0\quad \text{almost uniformly on }[0,\infty),
\end{equation}
where $u(t)$ is defined in \eqref{eq:dif-lim} and $w(\tau)$ oscillates according to a formulae
\begin{equation*}
w(\tau)=\sum_{n=1}^{\infty}e^{-(n\pi)^2\tau}a_n\cos{n\pi x},
\end{equation*}
where $a_n\in \RR$ is a parameter independent of $\tau$, for details see \cite[Eq.~(3.30)--(3.32)]{BFN:16}. The convergence results follow from the definition of $w$.
\end{proof}
Unlike in \autoref{flow-Kir-conv}, except for $t=0$, the macro-process defined in \eqref{eq:dif-lim} gives a good approximation of the micro-model.  Note, however, that the mass in the limit system is lumped by operator $\sP$ at each edge of the graph and only by considering a long time behaviour of the aggregated model \eqref{eq:dif-lim} we obtain the dynamics concentrated in the strong components like in \autoref{thm:dif-conv}. Formaly, define a projection $\Pi_0\colon  \rL^1(\mathcal{G}) \rightarrow  \rL^1(\mathcal{G})$, $\Pi_0 u:= \left(e\cdot u\right)\mathbf{1}$, where $e$ is the left eigenvector of $\sK^-$ associated with $\lambda=0$ chosen in the way that $e \cdot \mathbf{1}=1$. Now, $\Pi=\Pi_0\sP$. We can conclude that acceleration of a process of transmission distributes the vesicles uniformly in synaptic pools, which goes in line with intuition, since a slow rate of exchange between the pools (edges) traps the substance in them. Only by considering  a  sufficiently long time interval we obtain a uniform distribution in the all tree interconnected synaptic pools. We can drag the conclusion that, when the stimulus is sufficiently strong and repeats frequently  enough, then the response to it can decrease in time since there are no neurotransmitters in the so-called immediately available pool to serve it. The constructed model therefore reflects  a known biological phenomena called the habituation.

\section{Conclusion}
We have presented a short survey giving  some new insights of semigroup methods to the study of dynamical processes on metric graphs in a Banach space setting. In our approach, we do not treat boundary conditions in the junctions locally, but rather use  graph matrices to incorporate the structure of the whole graph. In this way, we are able to deduce certain qualitative properties of the solutions from the graph properties. The presented approach has a wide range of applications.

\end{document}